\documentclass[10pt]{amsart}
\usepackage[a4paper]{geometry}
\usepackage{stile_en}

\usepackage{graphicx}

\usepackage[utf8]{inputenc}
\usepackage[english]{babel}
\usepackage[T1]{fontenc}

\usepackage{type1cm}

\usepackage{amsmath, amsfonts, amssymb}
\usepackage[shortlabels]{enumitem}
\usepackage[all]{xy}

\usepackage{abbreviazioni}

\begin{document}

\title{Zeros of combinations of Euler products for $\sigma>1$}

\author{Mattia Righetti}

\address{Dipartimento di Matematica, Universit\`{a} di Genova}
\email{righetti@dima.unige.it}           

\date{}

\begin{abstract}
In this paper we consider Dirichlet series absolutely converging for $\sigma>1$ with an Euler product, natural bounds on the coefficients and satisfying orthogonality relations of Selberg type. Let $N\geq 1$, $F_1(s),...,F_N(s)$ be as above and $P(X_1,...,X_N)$ be a non-monomial polynomial with coefficients in the ring of $p$-finite Dirichlet series absolutely converging for $\sigma\geq 1$; then $P(F_1(s),\ldots,F_N(s))$ has infinitely many zeros for $\sigma>1$. Our result in particular applies to Artin $L$-functions, automorphic $L$-functions under the Ramanujan conjecture, and the elements of the Selberg class with polynomial Euler product under the Selberg orthonormality conjecture. This extends the work of Booker and Thorne \cite{booker}, who proved the same result for automorphic $L$-functions under the Ramanujan conjecture. Our proof avoids to use the properties of twists by Dirichlet characters, a key point in Booker and Thorne's proof, replacing them by results on the Dirichlet density of non-zero coefficients of $L$-functions of the above type.
\end{abstract}

\maketitle

\section{Introduction}
It is well known that linear combinations of $L$-functions may not satisfy the Riemann Hypothesis. For example, in 1936 Davenport and Heilbronn \cite{davenport1} proved that the Hurwitz zeta function
$$\zeta(s,a)=\sum_{n=0}^\infty \frac{1}{(n+a)^s}$$
has infinitely many zeros for $\sigma>1$ when $0<a<1$ is transcendental or rational with $a\neq \frac{1}{2}$. Note that, when $a=l/k$ is rational, $k^{-s}\zeta(s,l/k)$ may be written as a linear combination of Dirichlet $L$-functions $L(s,\chi)$, with $\chi$ varying among the Dirichlet characters mod $k$. The case $a$ irrational algebraic was settled successively by Cassels \cite{cassels}.

The idea of Davenport and Heilbronn, as pointed out by Bombieri and Ghosh \cite{bombierighosh}, was to apply Bohr's equivalence theorem to $\zeta(s,a)$. For a complete and general treatment of Bohr's equivalence theorem we refer to Chapter 8 of Apostol \cite{apostol}. Note that in \cite{apostol}, Bohr's equivalence theorem is stated for half-planes, but from the proof it is clear that the same holds for vertical strips.
\begin{theorem}[Bohr's equivalence theorem, {\cite{apostol}}]\label{theorem:bohr}
Let $F(s)=\sum_{n}a(n)\exp{-s\lambda_n}$ and $G(s)=\sum_n b(n) \exp{-s\lambda_n}$ be \emph{equivalent} (see \cite[\S8.7]{apostol}) general Dirichlet series (see \cite[\S8.2]{apostol}) with abscissa of absolute convergence $\sigma_a$. Then in any vertical strip $\sigma_a\leq\sigma_1<\sigma<\sigma_2$ the functions $F(s)$ and $G(s)$ take the same set of values.
\end{theorem}

In \cite{davenport1}, Davenport and Heilbronn explicitly find a Dirichlet series which is equivalent to $\zeta(s,a)$ and has a zero for $\sigma>1$, then by Bohr's equivalence theorem also $\zeta(s,a)$ has a zero for $\sigma>1$. Moreover, if we denote $s_0=\sigma_0+it_0$ this zero, by almost periodicity and Rouch\'{e}'s theorem, it is easy to verify that for any $\eps>0$
$$\#\{s=\sigma+it\mid \zeta(s,a)=0,\: \sigma_0-\eps<\sigma<\sigma_0+\eps,\: A<t<A+T\}\gg T$$
for any sufficiently large $T$, and all implied constants are independent of $A\in\R$.

For example, when $a=l/k$ is rational, we have to deal with the ordinary Dirichlet series $k^{-s}\zeta(s,l/k)$, for which we have the following statement.
\begin{theorem}[{\cite[Theorem 8.12]{apostol}}]\label{theorem:apostol}
Two ordinary Dirichlet series $F(s)=\sum_{n=1}^\infty a(n)n^{-s}$ and $G(s)=\sum_{n=1}^\infty b(n)n^{-s}$ are equivalent if and only if there exists a completely multiplicative function $\varphi(n)$ such that
\begin{enumerate}[a),leftmargin=.35in]
\item $|\varphi(p)|=1$ if $p$ is a prime dividing $n$ and $a(n)\neq 0$;
\item $b(n)=a(n)\varphi(n)$.
\end{enumerate}
\end{theorem}
\begin{remark}\label{remark:bohr_almost_per} Let be given a Dirichlet series $F(s)=\sum_{n=1}^\infty a(n)n^{-s}$ absolutely convergent for $\sigma>1$, and a completely multiplicative function $\varphi(n)$ with $|\varphi(n)|=1$ for every $n$; then, by Theorem \ref{theorem:apostol} and Bohr's equivalence theorem, for any $1\leq\sigma_1<\sigma_2$, the Dirichlet series
$$F^\varphi(s)=\sum_{n=1}^\infty \frac{a(n)\varphi(n)}{n^s}$$
takes the same set of values of $F(s)$ in $\sigma_1<\sigma<\sigma_2$. In particular, if $F^\varphi(s)$ has a zero in this vertical strip, so does $F(s)$. Moreover, as before, by Rouch\'{e}'s theorem and almost periodicity, in such a case one has
$$\#\{s=\sigma+it	\in\C\mid F(s)=0,\: \sigma_1<\sigma<\sigma_2,\: A<t<A+T\}\gg T$$
for any sufficiently large $T$, with all implied constants independent of $A\in\R$.
\end{remark}
When $a=l/k$ is rational, Davenport and Heilbronn \cite{davenport1} take $\varphi(n)$ defined at the primes $p$ as $\varphi(p)=i$ if the quadratic character $\lsym{p}{k}=-1$ and $\varphi(p)=1$ otherwise, and show that $\zeta^\varphi(s,l/k)$ has a zero for $\sigma>1$ (see \cite[Lemma 2]{davenport1}), then by Remark \ref{remark:bohr_almost_per}, it follows that $\zeta(s,l/k)$ has infinitely many zeros for $\sigma>1$.

Following Davenport and Heilbronn's method, Conrey and Ghosh \cite{conreyghosh} show\-ed that also the $L$-function associated to the square of Ramanujan's $\Delta$ cusp form has infinitely many zeros within its region of absolute convergence. Note that this $L$-function may be written as a linear combination of two $L$-functions associated to distinct degree-24 eigenforms.  

Recently, Kaczorowski and Kulas \cite{kulas} showed that, given $N\geq2$ pairwise nonequivalent Dirichlet characters $\chi_1,\ldots,\chi_N$ and $P_1,\ldots,P_N$ non-zero Dirichlet polynomials, the Dirichlet series
$$F(s)=\sum_{j=1}^N P_j(s)L(s,\chi_j)$$
has infinitely many zeros for $\frac{1}{2}<\sigma<1$, by using a strong joint universality property of Dirichlet $L$-functions.

Inspired by this work, Saias and Weingartner \cite{saias} proved that the same holds also for $\sigma>1$, by proving, through Brower fixed point theorem, a sort of ``weak joint universality property'' of Dirichlet $L$-functions for $\sigma>1$, i.e. 
\begin{quote}
given $R>1$ there exists $\eta>0$ such that for any $1<\sigma\leq 1+\eta$ and any $(z_1,\ldots,z_N)$, with $R^{-1}\leq |z_j|\leq R$ for all $j$, there exists $\varphi(n)$, completely multiplicative with $|\varphi(n)|=1$, such that $L^\varphi(\sigma,\chi_j)=z_j$, $j=1,\ldots, N$.
\end{quote}
In fact, writing $\varphi(p)=p^{-it_p}$, for some $t_p\in\R$, Brower fixed point theorem allows Saias and Weingartner to pass from trying to solve the Euler product system with $N$ equations and infinitely many variables
$$L^\varphi(\sigma,\chi_j)=\prod_p\left(1-\frac{\chi_j(p)}{p^{\sigma+it_p}}\right)^{-1}=z_j,\quad j=1,\ldots,N,$$
to the ``linear'' system with $N$ equations and infinitely many variables
\begin{equation}\label{eq:linear_system_Dirichlet}
\sum_p\frac{\chi_j(p)}{p^{\sigma+it_p}}=z_j,\quad j=1,\ldots,N,
\end{equation}
with the additional condition that $t_p$ must be continuous in the variables $(z_1,\ldots,z_N)$. Partitioning the primes into a finite number of (residue) classes, the system \eqref{eq:linear_system_Dirichlet} is reduced so to have a finite number of variables of modulus 1, which can be solved geometrically.\\
It is worth noting that this allows to generalize Davenport and Heilbronn method. Indeed, for any Dirichlet series $F(s)$ of the type studied by Kaczorowski and Kulas \cite{kulas}, Saias and Weingartner always find $\varphi(n)$ completely multiplicative with $|\varphi(n)|=1$ such that $F^\varphi(s)$ has a zero for $\sigma>1$. Then, as explained above, $F(s)$ has infinitely many zeros for $\sigma>1$.

Very recently, Booker and Thorne \cite{booker} refined Saias and Weingartner's technique and showed that all $L$-functions coming from unitary cuspidal automorphic representations of $\mathrm{GL}_r(\mathbb{A}_\Q)$, $r\geq 1$, share the same property, conditionally to the generalized Ramanujan conjecture at every finite place. Actually, the Ramanujan conjecture may be replaced with milder hypothesis, so that Booker and Thorne's result is unconditional for $r\leq 2$ \cite[Remark (3)]{booker}. Moreover, by cleverly using Hilbert's Nullstellensatz, they deduce from this property that not just linear, but also non-linear combinations of these $L$-functions, with Dirichlet polynomials as coefficients, have infinitely many zeros for $\sigma>1$, provided that there are at least two distinct non-zero terms.\\
As a downside, Booker and Thorne's proof still relies on the use of residue classes, limiting the sets of functions for which the proof of such a property is valid to those that are closed with respect to twists by Dirichlet characters. Although it is conjectured that this property holds for every $L$-function, for the moment being it may be of some interest to remove such an assumption. Moreover, for degree-two $L$-functions it is well known, by Weil's converse theorem, that the $L$-functions closed with respect to twists by Dirichlet characters are those coming from automorphic forms, provided that the functional equation of the twisted $L$-functions is of a given type. Hence a result which would not depend on such an assumption would have, in principle, a wider range of application. However, it must be said that Booker and Thorne \cite[Remark (4)]{booker} claim that, at the expense of making the proof more complicated, the use of residue classes could be avoided and that a similar result could be proven for an axiomatically-defined class of $L$-functions, such as the Selberg class.

In this paper we want to refine this technique by removing the use of residue classes, so that we can operate in a more general setting. Hence, let $\mathcal{E}$ be a class of complex functions $F(s)$ such that
\begin{enumerate}[(E1),leftmargin=.42in]
\vspace{-\topsep}
\item\label{hp:DS} $\displaystyle F(s)=\sum_{n=1}^\infty \frac{a_F(n)}{n^s}$, absolutely convergent for $\sigma>1$;
\item\label{hp:EP} $\displaystyle \log F(s)=\sum_{p} \log F_p(s) = \sum_p \sum_{k=1}^\infty\frac{b_F(p^k)}{p^{ks}}$, absolutely convergent for $\sigma>1$;
\item\label{hp:RC} there exists a constant $K_F$ such that $|a_F(p)|\leq K_F$ for every prime $p$;
\item\label{hp:HH} $\displaystyle \sum_p \sum_{k=2}^\infty\frac{|b_F(p^k)|}{p^{k}}<\infty$;
\item\label{hp:SC} for any pair of functions $F,G\in\mathcal{E}$ there exists $m_{F,G}\in\C$ such that

$ $\vspace{-10pt}
$$\sum_{p\leq x} \frac{a_F(p)\conj{a_G(p)}}{p} = (m_{F,G}+o(1))\log\log x,\quad x\rightarrow \infty,$$
with $m_{F,F}>0$.
\end{enumerate}
\begin{remark}\label{remark:non_zero}
If $F\in\mathcal{E}$, then $F(s)\neq 0$ for $\sigma>1$, by \ref{hp:EP}.
\end{remark}
\begin{definition}
We say that two functions $F,G\in\mathcal{E}$ are \emph{orthogonal} if $m_{F,G}=0$.
\end{definition}

In this setting we are able to prove the following ``weak joint universality property'', whose proof will be presented in Section \ref{section:representation}.
\begin{proposition}\label{proposition:main}
Let be given an integer $N\geq 1$, distinct functions $F_j(s)=\sum_{n}a_j(n)n^{-s}\in\mathcal{E}$, and real numbers $R,y\geq 1$. If $F_1, \ldots, F_N$ are pairwise orthogonal, then there exists $\eta>0$ such that for every $\sigma\in(1,1+\eta]$ we have
\begin{spliteq*}
&\left\{\left(\prod_{p>y}F_{1,p}(\sigma+it_p),\ldots,\prod_{p>y}F_{N,p}(\sigma+it_p)\right)  \mid t_p\in\R\right\}\\
&\qquad\qquad\qquad\qquad\supseteq\left\{(z_1,\ldots,z_N)\in\C^N\mid \frac{1}{R}\leq |z_j|\leq R\right\}.
\end{spliteq*}
\end{proposition}

This result is actually about the value distribution of $N$-uples of logarithms of Euler products, as it is clear from the proof. On this subject we mention that Nakamura and Pa\'{n}kowski \cite{nakamurapankowski} has obtained a similar result with similar hypotheses in the case of the logarithm of one Euler product, and its derivatives. We thank the referee for pointing out this article by Nakamura and Pa\'{n}kowski \cite{nakamurapankowski}, which was not yet available when this paper was written.

Let $\mathcal{P}$ be the set of primes of $\Z$. For $Q\subseteq\mathcal{P}$, we write $\langle Q\rangle=\{n\in\N\mid$ every prime factor of $n$ is in $Q\}$, then with $\mathcal{F}$ we denote the ring of \emph{$p$-finite Dirichlet series} (see \cite{kmp1}) absolutely convergent for $\sigma\geq 1$, i.e.
$$\mathcal{F}=\left\{\sum_{n\in\langle Q\rangle}\frac{a(n)}{n^s}\hbox{ abs. conv. for }\sigma\geq 1\mid Q\subseteq\mathcal{P}\hbox{ has finitely many elements}\right\},$$
which clearly contains all Dirichlet polynomials. Then, by adapting Booker and Thorne's proof of Theorem 1.2 of \cite{booker}, through Proposition \ref{proposition:main} one obtains the following result.

\begin{theorem}\label{theorem:main}
Fix an integer $N\geq 1$. For $j=1,\ldots,N$, let be given distinct functions $F_j(s)=\sum_{n}a_j(n)n^{-s}\in\mathcal{E}$. Suppose that $F_1,\ldots,F_N$ are pairwise orthogonal, then any polynomial $P\in \mathcal{F}[X_1,\ldots,X_N]$ either is a monomial or $P(F_1(s),\ldots,F_N(s))$ has infinitely many zeros for $\re(s)>1$. In the latter case there exists $\eta>0$ such that for any $1<\sigma_1<\sigma_2\leq 1+\eta$, we have
$$\#\{s=\sigma+it\mid P(F_1(s),\ldots,F_N(s))=0,\: \sigma_1<\sigma<\sigma_2,\: A<t<A+T\}\gg T$$
for any sufficiently large $T$, and all implied constants are independent of $A\in\R$.
\end{theorem}

In Section \ref{section:proof_main} we will give a proof of this theorem which is slightly different from a simple adaptation of Booker and Thorne's one \cite[\S4]{booker} in order to clear the underlying structure as presented in this introduction.

\begin{remark}
Note that the assumption $F_1(s), \ldots, F_N(s)$ pairwise orthogonal is necessary for Theorem \ref{theorem:main} to hold for \emph{any} polynomial $P$. In fact, take for example two orthogonal elements $F,G\in\mathcal{E}$, and consider $N=3$, $F_1=F^2$, $F_2=FG$, $F_3=G^2$, and $P=2X_2^3-X_1X_2X_3$. Then $P(F_1(s),F_2(s),F_3(s);s)=(F(s)G(s))^3$ which never vanishes for $\sigma>1$ by Remark \ref{remark:non_zero}, although $P$ is not a monomial.
\end{remark}

As a consequence of Theorem \ref{theorem:main} we have a partial result toward the following conjecture (see Bombieri and Ghosh \cite[p. 230]{bombierighosh})
\begin{quote}
The real parts of the zeros of a linear combination of two or more $L$-functions are dense in the interval $(1,\sigma^*)$, where $\sigma^*>1$ is the least upper bound of the real parts of such zeros. 
\end{quote}
Indeed we have the following result.
\begin{corollary}
Let be given an integer $N\geq 2$, pairwise orthogonal functions $F_1,\ldots,F_N\in\mathcal{E}$, and non-zero constants $c_1,\ldots,c_N\in\C$. Then there exists $\tilde{\sigma}$ such that $\{\sigma\in(1,\tilde{\sigma}]\mid \exists t\in\R$ s.t. $\sum_j c_j F_j(\sigma+it)=0\}$ is dense in $(1,\tilde{\sigma}]$.
\end{corollary}
\begin{proof} Apply Theorem \ref{theorem:main} to $P=\sum_{j=1}^N c_j X_j$, which clearly is not a monomial. Then, setting $\tilde{\sigma}=1+\eta$, the statement follows by the second part of Theorem \ref{theorem:main} taking, for any $\sigma\in (1,\tilde{\sigma}]$ and any $\eps>0$, $\sigma_1=\sigma-\eps$ and $\sigma_2=\sigma+\eps$.
\end{proof}

\subsection{Applications}

Here we show that Theorem \ref{theorem:main} may be applied in many cases.\vspace{0.2cm}

\noindent\textbf{Artin $L$-functions.} For an introduction on Artin's $L$-functions we refer to Chapter V of Neukirch \cite{neukirch}. Let $L(s,\rho,L/K)$ be the Artin $L$-function associated to the Galois extension of number fields $L/K$ with Galois group $G=\mathrm{Gal}(L/K)$, and to the representation $\rho$ of $G$. Note that \ref{hp:DS}, \ref{hp:EP}, \ref{hp:RC} and \ref{hp:HH} hold as an immediate result following from the definition, while \ref{hp:SC} follows from Chebotarev's Density Theorem (see \cite[Theorem 6.4]{neukirch}). In particular, by the orthogonality of characters, if $\rho_1$ and $\rho_2$ are both \emph{irreducible}, the corresponding $L$-functions are orthogonal (see, for example, \cite[Fact 3]{kmp2}).

\begin{corollary}\label{corollary:artin_L-functions}
Fix an integer $N\geq 1$. For $j=1,\ldots,N$, let be given Galois extensions $K_j$ over $\Q$ with Galois group $G_j$, and representations $(\rho_j,V_j)$ of $G_j$. Denote with $G$ the Galois group of $K_1\cdots K_N$ and suppose that the representations are all distinct and irreducible representations of $G$, then, if $P\in \mathcal{F}[X_1,\ldots,X_N]$ is not a monomial, $P(L(s,\rho_1,K_1/\Q),\ldots,L(s,\rho_N,K_N/\Q))$ has infinitely many zeros for $\sigma>1$.
\end{corollary}

\noindent\textbf{Automorphic $L$-functions.} For an introduction on $L$-series attached to unitary cuspidal automorphic representations of $\mathrm{GL}_r(\mathbb{A}_\Q)$ we refer to Rudnick and Sarnak \cite{rudnick}, and Iwaniec and Sarnak \cite{iwaniec}. Let $\pi=\otimes_p \pi_p$ be a unitary cuspidal automorphic representation of $\mathrm{GL}_r(\mathbb{A}_\Q)$, for some integer $r\geq 1$, and $L(s,\pi)$ be the associated $L$-function. It is an easy consequence of the definition that \ref{hp:DS} and \ref{hp:EP} hold for $L(s,\pi)$. On the other hand \ref{hp:RC}, \ref{hp:HH} and \ref{hp:SC} have not been yet proved in full generality, but they are known to hold under the Ramanujan conjecture: \ref{hp:RC} and \ref{hp:HH} follow immediately, while \ref{hp:SC}, as pointed out by Bombieri and Hejhal in \cite{bombierihejhal} and by Kaczorowski, Molteni and Perelli in \cite{kmp2}, follows from the properties of the Rankin--Selberg convolution (cf. Liu and Ye \cite{liu2} for a detailed proof of this fact).

We hence have the following, which is similar to Theorem 1.2 of Booker and Thorne \cite{booker}.

\begin{corollary}\label{corollary:automorphic_L-functions}
Fix an integer $N\geq 1$. For any $j=1,\ldots,N$, let be given a positive integer $r_j$ and a unitary cuspidal automorphic representation of $\mathrm{GL}_{r_j}(\mathbb{A}_\Q)$ with $L$-series $L(s,\pi_j)=\sum_{n=1}^\infty a_j(n)n^{-s}$. Suppose furthermore that $\pi_1,\ldots,\pi_N$ satisfy the generalized Ramanujan conjecture at all finite places (so that, in particular, $|a_j(p)|\leq r_j$ for all primes $p$ and $j=1,\ldots,N$) and are pairwise non-isomorphic. Then, if $P\in \mathcal{F}[X_1,\ldots,X_N]$ is not a monomial, $P(L(s,\pi_1),\ldots,L(s,\pi_N))$ has infinitely many zeros for $\sigma>1$.
\end{corollary}

\noindent\textbf{Selberg class.} For an introduction on the Selberg class we refer to the original paper of Selberg \cite{selberg}, and the surveys of Kaczorowski \cite{kaczorowski}, Kaczorowski and Perelli \cite{kaczper1}, and Perelli \cite{perelli1,perelli2}. The Selberg class $\selb$ is an axiomatically defined class of complex functions, introduced by Selberg \cite{selberg}, and we have that $F\in\selb$ satisfies \ref{hp:DS}, \ref{hp:EP} and \ref{hp:HH} as an easy consequence of the definition. However, in this setting \ref{hp:RC} and \ref{hp:SC} are not known, but they are expected to be true. For instance, \ref{hp:SC} corresponds to a deep conjecture for $\selb$, that is Selberg orthonormality conjecture (SOC) for \emph{primitive} elements (see for example \cite{perelli1} for an account on some of the interesting consequences which would follow). On the other hand, if we restrict to the subsemigroup $\selb^{\mathrm{poly}}$ of $\selb$ consisting of elements of $\selb$ with polynomial Euler product (see \cite{kaczper3} for an introduction on $\selb^{\mathrm{poly}}$), then \ref{hp:RC} follows immediately from the hypotheses. Hence, if we assume SOC we have the following result for $\selb^{\mathrm{poly}}$.

\begin{corollary}\label{corollary:selberg_class}
Fix an integer $N\geq 1$. For $j=1,\ldots,N$, let be given distinct primitive functions $F_j(s)=\sum_{n}a_j(n)n^{-s}\in\selb^{\mathrm{poly}}$. Suppose that SOC holds, then, if $P\in \mathcal{F}[X_1,\ldots,X_N]$ is not a monomial, $P(F_1(s),\ldots,F_N(s))$ has infinitely many zeros for $\sigma>1$.
\end{corollary}

As a final remark of this section we note that the elements of $\mathcal{E}$ are neither required to satisfy any functional equation nor to have a meromorphic continuation to the whole complex plane, thus Theorem \ref{theorem:main} may actually have a wider range of application than the examples given here (which conjecturally cover all $L$-functions), even though we are not aware of any example of such a class of Euler products.

\section{Densities of sets of primes}

The aim of this section is to show a result on the Dirichlet density of non-zero coefficients of $L$-functions. To this end we need some basic facts about densities of subsets of integers, but we weren't able to find any precise reference for them, although most of the followings may be easily deduced from Chapter III.1 of Tenenbaum \cite{tenenbaum}. 

\begin{definition}\label{def:natural_density}
Suppose $A\subseteq B\subseteq \N$, then we say that $A$ has \textbf{natural density} $d_B(A)$ in $B$ if
$$\lim_{x\rightarrow\infty}\frac{A(x)}{B(x)}=d_B(A),$$
where $A(x)=\#\{n\in A\mid n\leq x\}$ is the counting function.
Moreover we say that $A$ has \textbf{lower natural density} (resp. \textbf{upper natural density}) $\underline{d}_B(A)$ (resp. $\overline{d}_B(A)$) in $B$ if    
$$\liminf_{x\rightarrow\infty}\frac{A(x)}{B(x)}=\underline{d}_B(A)\quad (\hbox{resp. }\limsup= \overline{d}_B(A)).$$
\end{definition}

\begin{lemma}\label{lemma:subset_given_density}
Given any infinite subset $Q\subseteq \N$ and any $0\leq\alpha\leq1$, there exists a subset $A\subseteq Q$, such that $d_Q(A)=\alpha$.
\end{lemma}
\begin{proof}If $\alpha=0$, then we can take $A$ as any finite subset of $Q$. If $0<\alpha\leq 1$, let $p_n$ indicate the $n$-th element of $Q$, then take $A=\{p_{\lfloor \alpha^{-1} n \rfloor}\mid n\in \N\}$. Since $A(x)=\max \{n\mid p_{\lfloor \alpha^{-1} n \rfloor}\leq x\}$ and $p_{\lfloor \alpha^{-1} A(x) \rfloor}\in Q$, then $Q(x)\geq \lfloor \frac{A(x)}{\alpha} \rfloor$. On the other hand $x<p_{\lfloor \alpha^{-1} (A(x)+1) \rfloor}\in Q$, hence $Q(x)< \left\lfloor \frac{A(x)+1}{\alpha}\right\rfloor$. Therefore the result follows from the squeeze theorem.
\end{proof}

\begin{definition}\label{def:Dirichlet_density}
Suppose $A\subseteq B\subseteq \N$, then we say that $A$ has \textbf{Dirichlet density} (or \textbf{analytic density}) $\delta_B(A)$ in $B$ if
$$\lim_{\sigma\rightarrow 1^+}\frac{\sum_{n\in A}n^{-\sigma}}{\sum_{n\in B}n^{-\sigma}}=\delta_B(A),$$
Moreover we say that $A$ has \textbf{lower Dirichlet density} (resp. \textbf{upper Dirichlet density}) $\underline{\delta}_B(A)$ (resp. $\overline{\delta}_B(A)$) in $B$ if    
$$\liminf_{\sigma\rightarrow 1^+}\frac{\sum_{n\in A}n^{-\sigma}}{\sum_{n\in B}n^{-\sigma}}=\underline{\delta}_B(A)\quad (\hbox{resp. }\limsup= \overline{\delta}_B(A)).$$
\end{definition}

\begin{remark}\label{remark:denom}
Since it is well known that
$$\sum_{p}\frac{1}{p^\sigma}=-\log(\sigma-1) + O(1),\quad \sigma\rightarrow 1^+,$$
then for any $Q\subseteq\mathcal{P}$, we have
$$\delta_{\mathcal{P}}(Q)=\lim_{\sigma\rightarrow 1^+}\frac{\sum_{p\in Q}p^{-\sigma}}{-\log (\sigma-1)}.$$
Analogously for $\underline{\delta}_{\mathcal{P}}(A)$ and $\overline{\delta}_{\mathcal{P}}(A)$.
\end{remark}

There is a relation between natural density and Dirichlet density, that is

\begin{lemma}\label{lemma:Dirichlet_density}
Let $A\subseteq B\subseteq \N$ and suppose that $\sum_{n\in B}n^{-1}$ diverges, then
$$\underline{d}_B(A)\leq \underline{\delta}_B(A)\leq \overline{\delta}_B(A)\leq \overline{d}_B(A).$$
In particular, it follows that if $A$ has natural density $d_B(A)$ in $B$, then it has Dirichlet density $\delta_B(A)=d_B(A)$ in $B$.
\end{lemma}
\begin{proof}
We first observe that by partial summation we have
\begin{equation}\label{eq:part_summ}
\sum_{n\in A}\frac{1}{n^\sigma}=\sigma\int_{1^-}^\infty x^{-\sigma-1}A(x)dx.
\end{equation}
On the other hand, by definition we have that for any $\eps>0$ there exists $x_0$ such that
$$(\underline{d}_B(A)-\eps)B(x) < A(x) <(\overline{d}_B(A)+\eps)B(x),$$
for any $x>x_0$. Actually, there exists $M>0$ such that 
$$(\underline{d}_B(A)-\eps)B(x) - M < A(x) <(\overline{d}_B(A)+\eps)B(x) + M,$$
for any $x>0$. Hence, inserting these inequalities in \eqref{eq:part_summ}, we have
$$(\underline{d}_B(A)-\eps)\sum_{n\in B}\frac{1}{n^\sigma} - M<\sum_{n\in A}\frac{1}{n^\sigma}<(\overline{d}_B(A)+\eps)\sum_{n\in B}\frac{1}{n^\sigma} + M.$$
Dividing by $\sum_{n\in B}n^{-\sigma}$ and taking the $\liminf$ or the $\limsup$ as $\sigma\rightarrow1^+$, we get
$$\underline{d}_B(A)-\eps\leq \underline{\delta}_B(A)\leq \overline{\delta}_B(A)\leq \overline{d}_B(A)+\eps.$$
For the arbitrariness of $\eps$, we can make $\eps\rightarrow 0^+$ and we obtain the result.
\end{proof}

We now state some basic and general properties about $\limsup$ and $\liminf$ (see for example \cite[\S{II.5} Exercise 2]{amann}).

\begin{lemma}\label{lemma:trivial_proposition}
Given $f,g:\R\rightarrow[0,1]$ and a point $x_0\in\R\cup\{\pm\infty\}$, the following hold:
\begin{equation}\label{eq:inf-_-sup}
\liminf_{x\rightarrow x_0} [-f(x)]= -\limsup_{x\rightarrow x_0} f(x);
\end{equation}
\begin{spliteq}\label{eq:inf_inf_inf}
\liminf_{x\rightarrow x_0} [f(x)+g(x)]&\geq \liminf_{x\rightarrow x_0} f(x) + \liminf_{x\rightarrow x_0} g(x)\hbox{, and }\\
\liminf_{x\rightarrow x_0} [f(x)\cdot g(x)]&\geq \liminf_{x\rightarrow x_0} f(x) \cdot \liminf_{x\rightarrow x_0} g(x);
\end{spliteq}
\begin{spliteq}\label{eq:sup_sup_sup}
\limsup_{x\rightarrow x_0} [f(x)+g(x)]&\leq \limsup_{x\rightarrow x_0} f(x) + \limsup_{x\rightarrow x_0} g(x)\hbox{, and }\\
\limsup_{x\rightarrow x_0} [f(x)\cdot g(x)]&\leq \limsup_{x\rightarrow x_0} f(x) \cdot \limsup_{x\rightarrow x_0} g(x);
\end{spliteq}
\begin{equation}\label{eq:inf_sup_inf}
\liminf_{x\rightarrow x_0} [f(x)+g(x)]\leq \liminf_{x\rightarrow x_0} f(x) + \limsup_{x\rightarrow x_0} g(x);
\end{equation}
if $f(x)+g(x)=1$ for all $x\in\R$, then
\begin{equation}\label{eq:inf_1-sup}
\liminf_{x\rightarrow x_0} f(x) + \limsup_{x\rightarrow x_0} g(x) = 1;
\end{equation}
if $f(x)\leq g(x)$ for all $x\in\R$, then
\begin{equation}\label{eq:inf_inf_sup_sup}
\liminf_{x\rightarrow x_0} f(x)\leq \liminf_{x\rightarrow x_0} g(x),\hbox{ and}\quad\limsup_{x\rightarrow x_0} f(x)\leq \limsup_{x\rightarrow x_0} g(x).
\end{equation}
\end{lemma}

We will need the following lemma.

\begin{lemma}\label{lemma:log2dir}
Let $\{a(p)\}_{p\in\mathcal{P}}\subset \C$ be such that
\begin{equation}\label{eq:hp_SOC}
\sum_{p\leq x}\frac{|a(p)|^2}{p}\sim \kappa\log\log x,\quad x\rightarrow \infty,
\end{equation}
with $\kappa>0$. Then
\begin{equation}\label{eq:equiv_SOC}
\sum_{p\in\mathcal{P}}\frac{|a(p)|^2}{p^\sigma}\sim -\kappa\log(\sigma-1),\quad \sigma\rightarrow 1^+.
\end{equation}
\end{lemma}
\begin{proof}
By \eqref{eq:hp_SOC} we have that for any arbitrarily fixed $\eta>0$ there exists $x_0$ such that for any $x\geq x_0$ we have $L(x)=\sum_{p\leq x}\frac{|a(p)|^2}{p}\geq (\kappa-\eta)\log\log x$. Then, for any $\sigma>1$, by partial summation we get
\begin{spliteq*}
\sum_{p\in \mathcal{P}}\frac{|a(p)|^2}{p^\sigma} &= (\sigma-1)\int_1^\infty x^{-\sigma} L(x)dx\\
&\geq (\sigma-1)\int_1^\infty \frac{(\kappa-\eta)\log\log x}{x^\sigma}dx+O_\eta(\sigma-1)\\
&= (\kappa-\eta)\int_0^{\infty}\exp{-w} \log \left(\frac{w}{\sigma-1}\right) dw+O_\eta(\sigma-1)\\
&= -(\kappa-\eta)\log (\sigma-1) + O_\eta(\sigma-1).
\end{spliteq*}
Hence, by \eqref{eq:inf_inf_sup_sup}, we have
$$\liminf_{\sigma\rightarrow 1^+}\frac{\sum_{p\in\mathcal{P}}\frac{|a(p)|^2}{p^\sigma}}{-\log(\sigma-1)}\geq \kappa-\eta.$$
Analogously we obtain
$$\limsup_{\sigma\rightarrow 1^+}\frac{\sum_{p\in\mathcal{P}}\frac{|a(p)|^2}{p^\sigma}}{-\log(\sigma-1)}\leq \kappa+\eta.$$
Since the $\limsup$ and the $\liminf$ do not depend on $\eta$, which was arbitrarily chosen, we can take the limit for $\eta\rightarrow0^+$ and we obtain \eqref{eq:equiv_SOC}.
\end{proof}

We can now formulate the key lemma for the main result.

\begin{lemma}\label{lemma:positive_density}
Let $\{a(p)\}_{p\in\mathcal{P}}\subset \C$ be such that it satisfies \eqref{eq:hp_SOC} with $\kappa>0$. Suppose furthermore that there exists $M\geq \sqrt{\kappa}$ such that $|a(p)|\leq M$ for every prime $p$. Then, for any $\kappa-\sqrt{\kappa}< \gamma \leq\kappa$ the set $P_\gamma=\{p\in\mathcal{P} \mid |a(p)|\geq \kappa-\gamma\}$ has positive lower Dirichlet density
\begin{equation}\label{eq:lower_bound_density}
\underline{\delta}_\mathcal{P}(P_\gamma)\geq \frac{\kappa-(\kappa-\gamma)^2}{M^2-(\kappa-\gamma)^2}.
\end{equation}
\end{lemma}
\begin{proof}
Fix $\kappa-\sqrt{\kappa}< \gamma \leq\kappa$. Then, by hypothesis and Lemma \ref{lemma:log2dir} we have
\begin{spliteq*}
\kappa &=\liminf_{\sigma\rightarrow1^+} \frac{\sum_{p\in \mathcal{P}}\frac{|a(p)|^2}{p^\sigma}}{-\log (\sigma-1)}\\
& \stackrel{\eqref{eq:inf_inf_sup_sup}}{\leq} \liminf_{\sigma\rightarrow1^+} \left[\frac{M^2\sum_{p\in P_\gamma}\frac{1}{p^\sigma}+(\kappa-\gamma)^2\sum_{p\in (P_\gamma)^c}\frac{1}{p^\sigma}}{-\log (\sigma-1)}\right]\\
& \stackrel{\eqref{eq:inf_sup_inf}}{\leq} M^2\underline{\delta}_\mathcal{P}(P_\gamma)+(\kappa-\gamma)^2\overline{\delta}_\mathcal{P}((P_\gamma)^c)\\
& \stackrel{\eqref{eq:inf_1-sup}}{=} M^2\underline{\delta}_\mathcal{P}(P_\gamma) + (\kappa-\gamma)^2(1-\underline{\delta}_\mathcal{P}(P_\gamma)).
\end{spliteq*}
From this \eqref{eq:lower_bound_density} follows immediately, and it is easy to check that it is always positive.
\end{proof}

\begin{corollary}\label{corollary:positive_density}
Let be given an integer $N\geq 1$ and, for $j=1,\ldots,N$, distinct functions $F_j(s)=\sum_{n}a_j(n)n^{-s}\in\mathcal{E}$. If $F_1,\ldots,F_N$ are pairwise orthogonal, then there exists a positive constant $\delta$ such that for any vector $u=(u_1,\ldots,u_N)$ with $|u|=1$, the subset
$$Q_u=\left\{p\in\mathcal{P} \mid \abs{\frac{u_1a_1(p)}{\sqrt{m_{F_1,F_1}}}+\cdots+\frac{u_Na_N(p)}{\sqrt{m_{F_N,F_N}}}}\geq \frac{1}{2}\right\}$$
has positive lower Dirichlet density in $\mathcal{P}$ greater or equal than $\delta$.
\end{corollary}
\begin{proof}By \ref{hp:SC} and orthogonality, we have that $\left\{\frac{u_1a_1(p)}{\sqrt{m_{F_1,F_1}}}+\cdots+\frac{u_Na_N(p)}{\sqrt{m_{F_N,F_N}}}\right\}_{p}$ satisfies \eqref{eq:hp_SOC} with $\kappa=1$. Moreover by Cauchy-Schwarz and the triangle inequality we have that
$$\abs{\frac{u_1a_1(p)}{\sqrt{m_{F_1,F_1}}}+\cdots+\frac{u_Na_N(p)}{\sqrt{m_{F_N,F_N}}}}\stackrel{\ref{hp:RC}}{\leq} \sqrt{\frac{K_{F_1}^2}{m_{F_1,F_1}}+\cdots+\frac{K_{F_N}^2}{m_{F_N,F_N}}},$$
for every prime $p$. Since $Q_u$ coincides with the set $P_{\frac{1}{2}}$ of Lemma \ref{lemma:positive_density} applied to the sequence $\left\{\frac{u_1a_1(p)}{\sqrt{m_{F_1,F_1}}}+\cdots+\frac{u_Na_N(p)}{\sqrt{m_{F_N,F_N}}}\right\}_{p}$ with $\gamma=\frac{1}{2}$, we have
$$\underline{\delta}_\mathcal{P}(Q_u)\geq \frac{3}{4\max\left(1,\frac{K_{F_1}^2}{m_{F_1,F_1}}+\cdots+\frac{K_{F_N}^2}{m_{F_N,F_N}}\right)-1}=\delta.$$
\end{proof}

\begin{remark}\label{remark:remove_primes}
For any fixed $y>0$, denote with $\mathcal{P}_y=\{p\in\mathcal{P}\mid p>y\}$. Since $y$ is fixed, $\mathcal{P}_y$ has density 1 in $\mathcal{P}$, thus all of the above still hold if we replace $\mathcal{P}$ with $\mathcal{P}_y$.
\end{remark}

\section{Proof of Proposition \ref{proposition:main}}
\label{section:representation}

Fixed an integer $N\geq 1$, we denote with $\mathrm{GL}_N(\C)$ the topological group of invertible matrices $N\times N$ with complex coefficients. For any $R>0$ we further set
$$D_N(R)=\{z=(z_1,\ldots,z_N)\in \C^N\mid |z_j|\leq R\},\quad D_N=D_N(1),$$
$$T_N(R)=\{z\in\C^N\mid |z_j|=R\},\quad T_N=T_N(1),$$
and we recall Proposition 3.2 of Booker and Thorne \cite{booker}.

\begin{proposition}[Booker--Thorne]\label{proposition:unitary_combination_of_matrices}
Let be given a compact $K\subseteq \mathrm{GL}_N(\C)$. Then there is a number $m_0>0$ such that for every $m\geq m_0$ and all $(g_1,\ldots,g_m)$ in $K^m$, there are continuous functions $f_1,\ldots,f_m: D_N\rightarrow T_N$ such that $\displaystyle\sum_{i=1}^m g_if_i(z)=z$ for all $z\in D_N$.
\end{proposition}

This is a fundamental ingredient for the proof of Proposition \ref{proposition:main}, together with the results of the previous section, as it is fundamental for Proposition 3.1 of \cite{booker}. 

We will also need a result on the conditional convergence of series, so let be $\{\omega_n\}_{n\in\N}$ be a sequence with values in $\{\pm 1\}$. Note that on the space of such sequences it is possible to put a probability measure (see for example \cite[\S 1.2--1.6]{kac}). In this setting, the following result is due to Rademacher, Paley and Zigmund (see \cite[\S 2.5--2.6]{kac}).

\begin{theorem}[Rademacher--Paley--Zygmund]\label{theorem:RPZ} Let $\{a_n\}_{n\in\N}\subset\R$. The following are equivalent:
\begin{enumerate}[a),leftmargin=.35in]
\item The probability that $\sum_n \omega_n a_n$ converges is 1.
\item $\sum_n |a_n|^2 < \infty$.
\end{enumerate}
\end{theorem}

\begin{remark}\label{remark:RPZ_gen}
This theorem clearly may be extended to the case $\{a_n\}\subset\C$ taking the real and imaginary parts, and in general to the case $\{a_n\}\subseteq\R^N$ or $\C^N$, $N\geq 1$, since a finite intersection of measure 1 sets still has measure 1. Actually with an analogous argument it can be proven for $\{a_n\}$ belonging to any separable Hilbert space.
\end{remark}

Therefore we have the following.

\begin{corollary}\label{corollary:uniform_bound}
Let be given an integer $N\geq 1$, for $j=1,\ldots,N$ distinct elements $F_j(s)=\sum_{n}a_j(n)n^{-s}\in\mathcal{E}$, and suppose that they are pairwise orthogonal. Then, for any infinite subset $Q\subseteq \mathcal{P}$ there exist $\{\omega_p\}_{p\in Q}\subseteq\{\pm 1\}$ such that the vectors
$$v(\sigma)=\left(\sum_{p\in Q} \frac{a_1(p)\omega_p}{p^\sigma},\ldots,\sum_{p\in Q} \frac{a_n(p)\omega_p}{p^\sigma}\right)$$
are uniformly bounded for $\sigma\in [1,+\infty)$.
\end{corollary}
\begin{proof}
Call $v_p(\sigma)$ the vector 
$$v_p(\sigma)=\left(\frac{a_1(p)}{p^\sigma},\ldots,\frac{a_n(p)}{p^\sigma}\right),$$
so that $v(\sigma)=\sum_{p\in Q} \omega_p v_p(\sigma)$, with $\{\omega_p\}$ still to be chosen. Note that by partial summation we have 
$$\sum_{p\in Q}|v_p(1)|^2\leq \int_{2^-}^\infty \frac{1}{x^2} \sum_{j=1}^N\sum_{p\leq x}\frac{|a_j(p)|^2}{p}dx\stackrel{\ref{hp:SC}}{\ll} \int_{2^-}^\infty \frac{\log\log x}{x^2}dx<+\infty.$$
Therefore, by the previous theorem and remark, we can surely find $\{\omega_p\}_{p\in Q}\subseteq\{\pm 1\}$ such that $v(1)$ is convergent. Moreover, again by partial summation, for any $\sigma\geq 1$ and any $j=1,\ldots,N$, we have
$$\abs{\sum_{p\in Q} \frac{a_j(p)\omega_p}{p^\sigma}}\leq 2\sup_{x>1} \abs{\sum_{\substack{p\in Q\\p\leq x}} \frac{a_j(p)\omega_p}{p}},$$
which is finite for the above choice of $\omega_p$ since the series converges. Hence, for any $\sigma\geq 1$, we have the uniform bound
$$|v(\sigma)|^2\leq 4\sum_{j=1}^N \sup_{x>1} \abs{\sum_{\substack{p\in Q\\p\leq x}} \frac{a_j(p)\omega_p}{p}}^2.$$
\end{proof}

We can now state and prove the key result of this section.

\begin{proposition}\label{proposition:linear_system}
Let be given a positive integer $N$, real numbers $\rho\geq1$ and $y>0$, and for $j=1,\ldots,N$, distinct elements $F_j(s)=\sum_{n}a_j(n)n^{-s}\in\mathcal{E}$. Suppose that $F_1,\ldots,F_N$ are pairwise orthogonal, then there exists $\eta>0$ such that for any $\sigma\in(1,1+\eta]$ we can find continuous functions $t_p: D_N(\rho)\rightarrow \R$ for each prime $p>y$ such that
\begin{equation}\label{eq:linear_system}
\sum_{p>y}\frac{a_j(p)}{p^{\sigma+it_p(z)}}=z_j,\quad j=1,\ldots,N,
\end{equation}
for any $z\in D_N(\rho)$.
\end{proposition}
\begin{proof}The proof is an adaptation of that of Proposition 3.3 of Booker and Thorne \cite{booker}: as in \cite{booker} we want to construct inductively matrices belonging to a compact of $\mathrm{GL}_N(\C)$; the main difference is in the construction of these matrices and the fact that we take twice the matrices because we have a ``remainder'' term which we have to deal with. Hence, let be $m$ the integer $m_0$ obtained by Proposition \ref{proposition:unitary_combination_of_matrices} for the compact 
$$K=\left\{g\in\mathrm{GL}_N(\C)\mid \norma{g}\leq 2\delta \sqrt{N\sum_{j=1}^N K_{F_j}^2},\:\abs{\det g}\geq \left(\dfrac{\delta^2}{8}\right)^N\prod_{j=1}^N\sqrt{m_{F_j,F_j}}\right\},$$
where $\norma{g}^2=\sum_{i,j=1}^N |g_{i,j}|^2$, $\delta$ is given by Corollary \ref{corollary:positive_density}, $K_{F_j}$ by \ref{hp:RC}, and $m_{F_j,F_j}$ by \ref{hp:SC}. Now we want to construct inductively $2m$ matrices, namely $g_1,\ldots,g_{2m}$, all belonging to $K$.\\
For any fixed $i\in\{1,\ldots,2m\}$, take $u$ to be any vector in $\C^N$ with $|u|=1$, and $S_{i1}$ a subset of
$$\tilde{Q}_u=Q_u\backslash \bigcup_{j<i} (S_{j1}\cap Q_u)$$
with density $\delta_{\tilde{Q}_u}(S_{i1})=\displaystyle\frac{\delta}{(2mN)^2\overline{\delta}_{\mathcal{P}_y}(\tilde{Q}_u)}$ in $\tilde{Q}_u$ (for $i=1$, we take $\tilde{Q}_u=Q_u$), where $Q_u$ is the set of primes defined in Corollary \ref{corollary:positive_density}. We know that $S_{i1}$ exists by Lemma \ref{lemma:subset_given_density} since $\overline{\delta}_{\mathcal{P}_y}(\tilde{Q}_u)$ is greater or equal than
\begin{spliteq*}
\underline{\delta}_{\mathcal{P}_y}(\tilde{Q}_u) &\geq \underline{\delta}_{\mathcal{P}_y}(Q_u)-\sum_{j<i}\overline{\delta}_{\mathcal{P}_y}(Q_u\cap S_{j1}) \geq  \delta-\sum_{j<i}\overline{\delta}_{\mathcal{P}_y}(S_{j1})\\
&\geq \delta-(i-1)\frac{\delta}{(2mN)^2} \geq \frac{\delta}{2}.
\end{spliteq*}
Note that we have used the fact $\overline{\delta}_{\mathcal{P}_y}(S_{j1})=\delta_{\tilde{Q}_u}(S_{j1})\overline{\delta}_{\mathcal{P}_y}(\tilde{Q}_u)=\frac{\delta}{(2mN)^2}$.\\
Now, let $v_{i,1}$ be the column vector
$$v_{i,1}=\left(\frac{(2mN)^2}{\sum_{p>y}p^{-\sigma}}\sum_{p\in S_{i1}} \frac{a_j(p)\eps_p}{\sqrt{m_{F_j,F_j}}p^\sigma}\right)_{j=1,\ldots,N},$$
where $\eps_p\in T_1$ is such that $\eps_p\left(\frac{u_1a_1(p)}{\sqrt{m_{F_1,F_1}}}+\cdots+\frac{u_Na_N(p)}{\sqrt{m_{F_N,F_N}}}\right)>0$.

By induction, for $i\in\{1,\ldots,2m\}$ and $k\in\{2,\ldots,N\}$, we define the subset of primes $S_{ik}$ and the column vectors $v_{i,k}$ as follows. Take $u$ such that $|u|=1$ and $\conj{u}$ is orthogonal to the vector space generated by $v_{i,1},\ldots,v_{i,k-1}$. Then take $S_{ik}$ to be a subset of
$$\tilde{Q}_u=Q_u\backslash \left(\bigcup_{j<i}\bigcup_{\ell=1}^N (S_{j\ell}\cap Q_u)\cup \bigcup_{\ell=1}^{k-1} (S_{i\ell}\cap Q_u)\right)$$
with density $\delta_{\tilde{Q}_u}(S_{ik})=\displaystyle\frac{\delta}{(2mN)^2\overline{\delta}_{\mathcal{P}_y}(\tilde{Q}_u)}$ in $\tilde{Q}_u$. As before we know that $S_{ik}$ exists by Lemma \ref{lemma:subset_given_density} since $\overline{\delta}_{\mathcal{P}_y}(\tilde{Q}_u)$ is greater or equal than
\begin{spliteq}\label{eq:lower_bound_Q_tilde}
\underline{\delta}_{\mathcal{P}_y}(\tilde{Q}_u)& \geq \underline{\delta}_{\mathcal{P}_y}(Q_u)-\sum_{j< i}\sum_{\ell=1}^N\overline{\delta}_{\mathcal{P}_y}(S_{j\ell})-\sum_{\ell=1}^{k-1}\overline{\delta}_{\mathcal{P}_y}(S_{i\ell})\\
&\geq \delta-[(i-1)N+k-1]\frac{\delta}{(2mN)^2} \geq \frac{\delta}{2}.
\end{spliteq}
Then we define $v_{i,k}$ as
$$v_{i,k}=\left(\frac{(2mN)^2}{\sum_{p>y}p^{-\sigma}}\sum_{p\in S_{ik}} \frac{a_j(p)\eps_p}{\sqrt{m_{F_j,F_j}}p^\sigma}\right)_{j=1,\ldots,N},$$
where $\eps_p\in T_1$ is such that $\eps_p\left(\frac{u_1a_1(p)}{\sqrt{m_{F_1,F_1}}}+\cdots+\frac{u_Na_N(p)}{\sqrt{m_{F_N,F_N}}}\right)>0$.\\
Finally we set for every $i\in\{1,\ldots,2m\}$
$$g_i=\begin{pmatrix}\sqrt{m_{F_1,F_1}}&0 & \cdots&0 \\ 0 & \ddots &\ddots & \vdots\\ \vdots &\ddots &\ddots & 0 \\ 0 &\cdots & 0 & \sqrt{m_{F_N,F_N}}\end{pmatrix}(v_{i,1}|\ldots|v_{i,N})\in\mathrm{Mat}_{N\times N}(\C).$$

We need to show that the matrices $g_i$, $i=1,\ldots,2m$, belong to $K$. To this end we follow Booker and Thorne's method (see the proof of Proposition 3.3 of \cite{booker}), using the results of the previous section on densities of sets of primes.

To bound $\norma{g_i}$, we note that every coefficient of $g_i$ satisfies
$$\abs{\frac{(2mN)^2}{\sum_{p>y}p^{-\sigma}}\sum_{p\in S_{ik}} \frac{a_j(p)\eps_p}{p^\sigma}}\leq \frac{(2mN)^2}{\sum_{p>y}p^{-\sigma}}\sum_{p\in S_{ik}} \frac{\abs{a_j(p)}}{p^\sigma}\leq (2mN)^2 K_{F_j}\frac{\displaystyle\sum_{p\in S_{ik}}p^{-\sigma}}{\displaystyle\sum_{p>y}p^{-\sigma}}.$$
Since
$$\overline{\delta}_{\mathcal{P}_y}(S_{ik})=\delta_{\tilde{Q}_u}(S_{ik})\overline{\delta}_{\mathcal{P}_y}(\tilde{Q}_u)= \frac{\delta}{(2mN)^2},$$
there exists $\eta>0$ such that
$$\frac{\sum_{p\in S_{ik}}p^{-\sigma}}{\sum_{p>y}p^{-\sigma}}\leq 2\:\overline{\delta}_{\mathcal{P}_y}(S_{ik})= \frac{2\delta}{(2mN)^2}$$
for any $\sigma \in (1,1+\eta]$. Hence $\norma{g_i}\leq 2\delta \sqrt{N\sum_j K_{F_j}^2}$, $i=1,\ldots,2m$.

To bound $\abs{\det g_i}$, observe that, for any $k\in\{1,\ldots,N\}$, we have
\begin{spliteq*}
|v_{i,k}-\mathrm{proj}_{\langle v_{i,1},\ldots,v_{i,k-1}\rangle}v_{i,k}|&\geq \left\langle v_{i,k}-\mathrm{proj}_{\langle v_{i,1},\ldots,v_{i,k-1}\rangle}v_{i,k},\conj{u}\right\rangle \\
&= \left\langle v_{i,k},\conj{u}\right\rangle \\
&= \frac{(2mN)^2}{\sum_{p>y}p^{-\sigma}}\sum_{p\in S_{ik}} \frac{\abs{\frac{u_1a_1(p)}{\sqrt{m_{F_1,F_1}}}+\cdots+\frac{u_Na_N(p)}{\sqrt{m_{F_N,F_N}}}}}{p^\sigma}\\
&\geq \frac{(2mN)^2}{2}\frac{\sum_{p\in S_{ik}}p^{-\sigma}}{\sum_{p>y}p^{-\sigma}},
\end{spliteq*}
where $u$ is the norm-one vector used to construct $S_{ik}$ and, for the last step, we have used the fact that $S_{ik}\subseteq Q_u$. Reducing $\eta$ if necessary, we have that
\begin{spliteq*}
\frac{\sum_{p\in S_{ik}}p^{-\sigma}}{\sum_{p>y}p^{-\sigma}}&\geq \frac{1}{2}\underline{\delta}_{\mathcal{P}_y}(S_{ik})=\frac{1}{2}\delta_{\tilde{Q}_u}(S_{ik})\underline{\delta}_{\mathcal{P}_y}(\tilde{Q}_u)\\
&=\frac{1}{2}\frac{\delta}{(2mN)^2\overline{\delta}_{\mathcal{P}_y}(\tilde{Q}_u)}\underline{\delta}_{\mathcal{P}_y}(\tilde{Q}_u)\stackrel{\eqref{eq:lower_bound_Q_tilde}}{\geq} \frac{\delta^2}{4(2mN)^2}
\end{spliteq*}
for any $\sigma\in(1,1+\eta]$. Hence for any $\sigma\in (1,1+\eta]$ we have
$$|v_{i,k}-\mathrm{proj}_{\langle v_{i,1},\ldots,v_{i,k-1}\rangle}v_{i,k}|\geq \frac{\delta^2}{8},$$
and, by Gram–Schmidt orthogonalization, we obtain
$$|\det g_i|\geq \left(\dfrac{\delta^2}{8}\right)^N\displaystyle\prod_{j=1}^N\sqrt{m_{F_j,F_j}}.$$

Now, we define $S=\mathcal{P}_y\backslash \bigcup_{i=1}^{2m}\bigcup_{k=1}^N S_{ik}$ and we call $w_\sigma$ the ``remainder'' term, i.e. the column vector
$$w_\sigma=\left(\frac{(2mN)^2}{\sum_{p>y}p^{-\sigma}}\sum_{p\in S} \frac{a_j(p)\eps_p}{p^\sigma}\right)_{j=1,\ldots,N},$$
where $\{\eps_p\}_{p\in S}\subseteq\{\pm 1\}\subseteq T_1$ are chosen so that $w_\sigma$ is uniformly bounded by a constant $C\geq 1$ for $\sigma\in[1,1+\eta]$: we know that these exist by Corollary \ref{corollary:uniform_bound}.

In the following, we adapt Booker and Thorne's idea of applying Proposition \ref{proposition:unitary_combination_of_matrices} to the matrices just constructed. In fact, since we have the ``remainder'' term $w_\sigma$, we apply twice Proposition \ref{proposition:unitary_combination_of_matrices}, first to the first $m$ matrices and then to the remaining $m$ matrices to deal with $w_\sigma$.\\
Reducing again $\eta$ if necessary, we may suppose that $\sum_{p>y} p^{-\sigma}\geq \rho C(2mN)^2$ for any $\sigma\in (1,1+\eta]$. We fix such a $\sigma$ and we apply Proposition \ref{proposition:unitary_combination_of_matrices} to $(g_1,\ldots,g_m)\in K^m$, i.e. there exist continuous functions $f_1,\ldots,f_m: D_N\rightarrow T_N$ such that $\sum_{i=1}^m g_if_i(\tau)=\tau$ for all $\tau\in D_N$. Applying Proposition \ref{proposition:unitary_combination_of_matrices} to $(g_{m+1},\ldots,g_{2m})\in K^m$ we obtain $f_{m+1},\ldots,f_{2m}:D_N\rightarrow T_N$ such that $\sum_{i=m+1}^{2m} g_if_i(-w_\sigma)=-w_\sigma$. Summing up and setting for any $\tau\in D_N$ (for any fixed choice of a branch of the logarithm)
$$\theta_p(\tau) = \left\{
\begin{array}{ll}
-\log(\eps_p f_i(\tau)_k)/\log p & p\in S_{ik},\: i=1,\ldots,m,\: k=1,\ldots,N\\
-\log(\eps_p f_i(-w_\sigma)_k)/\log p & p\in S_{ik},\: i=m+1,\ldots,2m,\: k=1,\ldots,N\\
-\log(\eps_p)/ \log p & p\in S
\end{array}\right.$$
we have
\begin{spliteq*}
\sum_{p>y}\frac{a_j(p)}{p^{\sigma+i\theta_p(\tau)}}&=\sum_{i=1}^m \sum_{k=1}^N \sum_{p\in S_{ik}}\frac{a_j(p)\eps_p}{p^\sigma}f_i(\tau)_k\\
&\qquad\qquad+\sum_{i=m+1}^{2m} \sum_{k=1}^N \sum_{p\in S_{ik}}\frac{a_j(p)\eps_p}{p^\sigma}f_{i}(-w_\sigma)_{k}+\sum_{p\in S} \frac{a_j(p)\eps_p}{p^\sigma}\\
&=\frac{\sum_{p>y}p^{-\sigma}}{(2mN)^2}(\tau_j-w_{\sigma,j}+w_{\sigma,j})=\frac{\sum_{p>y}p^{-\sigma}}{(2mN)^2}\tau_j,
\end{spliteq*}
for $j=1,\ldots,N$. Since $\sum_{p>y} p^{-\sigma}\geq \rho (2mN)^2$, we can substitute $\tau = \frac{(2mN)^2}{\sum_{p>y}p^{-\sigma}}z$ for any $z\in D_N(\rho)$. Writing $t_p(z)=\theta_p\left(\frac{(2mN)^2}{\sum_{p>y}p^{-\sigma}} z\right)$, we obtain \eqref{eq:linear_system}.
\end{proof}

We adapt Lemma 2 of \cite{saias} and Proposition 3.1 of \cite{booker} for the class $\mathcal{E}$ as follows.

\begin{proposition}[Saias--Weingartner--Booker--Thorne]\label{proposition:SWBT} Let be given a positive integer $N$, for $j=1,\ldots,N$, distinct functions $F_j(s)=\sum_n a_j(n)n^{-s}\in\mathcal{E}$, and real numbers $R,y\geq 1$. Moreover, suppose that for any given $\rho \geq 1$ there exists $\eta>0$ such that for any $\sigma \in (1,1+\eta]$ there are continuous functions $t_p:D_N(\rho)\rightarrow \R$, for any prime $p>y$, satisfying
$$\sum_{p>y}\frac{a_j(p)}{p^{\sigma+it_p(z)}}=z_j,\quad j=1,\ldots,N,$$
for any $z=(z_1,\ldots,z_N)\in D_N(\rho)$. Then for all $\sigma\in (1,1+\eta]$ we have
\begin{spliteq*}
&\left\{\left(\prod_{p>y}F_{1,p}(\sigma+it_p),\ldots,\prod_{p>y}F_{N,p}(\sigma+it_p)\right)  \mid t_p\in\R\right\}\\
&\qquad\qquad\qquad\qquad\supseteq\left\{(z_1,\ldots,z_N)\in\C^N\mid \frac{1}{R}\leq |z_j|\leq R\right\}.
\end{spliteq*}
\end{proposition}

It is clear that, by Proposition \ref{proposition:linear_system}, if $F_j\in\mathcal{E}$, $j=1,\ldots,N$, are pairwise orthogonal, then they satisfy the hypotheses of Proposition \ref{proposition:SWBT}. We have thus proven Proposition \ref{proposition:main}.

\section{Proof of Theorem \ref{theorem:main}}
\label{section:proof_main}

By hypothesis we have a polynomial 
$$P(X_1,\ldots,X_N; s) = \sum_{i=1}^M D_i(s)\prod_{j=1}^N X_j^{\alpha_{ij}},$$
with $D_i(s)=\sum_{n\in\langle Q_i\rangle} c_i(n)n^{-s}\in\mathcal{F}$ not identically zero and $\alpha_{ij}\in\N\cup\{0\}$. If we write $F(s)=P(F_1(s),\ldots,F_N(s);s)$, then clearly $F(s)$ is a Dirichlet series absolutely convergent for $\sigma>1$; thus, by Remark \ref{remark:bohr_almost_per}, we just need to prove that there exist $\sigma>1$ and $\varphi(n)$, completely multiplicative with $|\varphi(n)|=1$, such that $F^\varphi(\sigma)=0$.
Since $\varphi$ must be completely multiplicative, it is sufficient to define its values only on the primes. Moreover, since we must have $|\varphi(p)|=1$, we write $\varphi(p)=\exp{-it_p}$, with $t_p\in\R$ (yet to be determined), for every prime $p$.\\
Let $y$ be a fixed (non-integral) real number such that $\bigcup_{i=1}^M Q_i \subseteq \{p\in\mathcal{P}\mid p\leq y\}.$ Then consider the polynomial
\begin{spliteq*}
Q(X_1,\ldots,X_N;s)&=P\left(\prod_{p\leq y}F_{1,p}(s)X_1,\ldots,\prod_{p\leq y}F_{N,p}(s)X_N;s\right)\\
&=\sum_{i=1}^M \tilde{D}_i(s)\prod_{j=1}^N X_j^{\alpha_{ij}}.
\end{spliteq*}
Note that the coefficients $\tilde{D}_i(s)$ belong to $\mathcal{F}$, indeed they are clearly $p$-finite, while the absolute convergence for $\sigma=1$ comes from \ref{hp:HH} and the fact that the sum is over a finite number of primes.

To study the zeros of the polynomial $Q$ we use the following two lemmas of Booker and Thorne \cite{booker}.

\begin{lemma}[{\cite[Lemma 2.4]{booker}}]\label{lemma:zero_pol}
Let $P\in\C[x_1,\ldots,x_n]$. Suppose that every solution to the equation $P(x_1,\ldots,x_n)=0$ satisfies $x_1\cdots x_n=0$, then P is a monomial.
\end{lemma}

\begin{lemma}[{\cite[Lemma 2.5]{booker}}]\label{lemma:close_zero}
Let $P\in\C[x_1,\ldots,x_n]$ and suppose that $y\in\C^n$ is a zero of $P$. Then, for any $\eps>0$ there exists $\delta>0$ such that any polynomial $Q\in\C[x_1,\ldots,x_n]$, obtained by changing the non-zero coefficients of $P$ of at most $\delta$ each, has a zero $z\in\C^n$ with $\abs{y-z}<\eps$.
\end{lemma}

Since the $p$-finite Dirichlet series $\tilde{D}_1(s),\ldots,\tilde{D}_N(s)$ are absolutely convergent for $\sigma\geq 1$, they are holomorphic in the half-plane $\sigma>1$ and extend with continuity on the line $\sigma=1$. Applying the maximum modulus principle to the function $\tilde{D}_1(s)\cdots\tilde{D}_N(s)$, which is not identically zero, we see that necessarily there exists $t_0\in\R$ such that $\tilde{D}_1(1+it_0),$ $\ldots,$ $\tilde{D}_N(1+it_0)$ are all non-zero. Therefore, applying Lemma \ref{lemma:zero_pol} to $Q(X_1,\ldots,X_N;1+it_0)$, we have that either $M=1$ or there exist $x_1,\ldots,x_N\in\C$, all non-zero, such that $Q(x_1,\ldots,x_N;1+it_0)=0$. Since in the first case we would have that $P$ is a monomial, we suppose that we are in the second case and we take $t_p=t_0$ for every prime $p\leq y$.

Let be $R\geq 2$ such that $\frac{2}{R}\leq |x_j|\leq \frac{R}{2}$, then, applying Lemma \ref{lemma:close_zero} with $\eps=\frac{1}{R}$, we obtain that there exists $\eta>0$ such that for any $\sigma\in(1,1+\eta]$ there exists $(z_1(\sigma),\ldots,z_N(\sigma))\in\C^N$ such that $Q(z_1(\sigma),\ldots,z_N(\sigma);\sigma+it_0)=0$ and $\frac{1}{R}\leq |z_j(\sigma)|\leq R$ for every $j$.

By Proposition \ref{proposition:main} for $R$ and $y$, we have that, possibly reducing $\eta$, for any $\sigma\in(1,1+\eta]$ there exist $t_p\in\R$ for every prime $p>y$ such that
$$z_j(\sigma)=\prod_{p>y}F_{j,p}(\sigma+it_p)=\prod_{p>y}F_{j,p}^\varphi(\sigma),\quad j=1,\ldots,N.$$
Hence, for any $\sigma\in(1,1+\eta]$ we have found $\varphi(n)$ (which depends on $\sigma$) such that
\begin{spliteq*}
F^\varphi(\sigma)&=\sum_{i=1}^M D_i^\varphi(\sigma)\prod_{j=1}^N F^\varphi_j(\sigma)^{\alpha_{ij}} = \sum_{i=1}^M \tilde{D}_i^\varphi(\sigma)\prod_{j=1}^N\prod_{p>y} F^\varphi_{j,p}(\sigma)^{\alpha_{ij}}\\
&=\sum_{i=1}^M \tilde{D}_i(\sigma+it_0)\prod_{j=1}^Nz_{j}(\sigma)^{\alpha_{ij}}=0.
\end{spliteq*}
The last part of the theorem follows, as we already said in Remark \ref{remark:bohr_almost_per}, by Rouch\'{e}'s theorem and almost periodicity.

\begin{acknowledgements}
We would like to express our sincere gratitude to Prof. Alberto Perelli and to Prof. Giuseppe Molteni for helpful discussions and valuable suggestions. We would also like to thank the referee for detecting some inaccuracies and suggesting improvements in the presentation.
\end{acknowledgements}

\bibliographystyle{amsplain}      
\bibliography{biblio}   

\providecommand{\bysame}{\leavevmode\hbox to3em{\hrulefill}\thinspace}
\providecommand{\MR}{\relax\ifhmode\unskip\space\fi MR }
\providecommand{\MRhref}[2]{%
  \href{http://www.ams.org/mathscinet-getitem?mr=#1}{#2}
}
\providecommand{\href}[2]{#2}
\begin{thebibliography}{10}

\bibitem{amann}
Herbert Amann and Joachim Escher, \emph{Analysis {I}. {T}ransl. from the
  {G}erman by {Gary Brookfield}}, 3rd ed., Birkh\"auser, Basel, 2006.

\bibitem{apostol}
Tom~M. Apostol, \emph{Modular functions and dirichlet series in number theory},
  2nd ed., Springer-Verlag, 1990.

\bibitem{bombierighosh}
Enrico Bombieri and Amit Ghosh, \emph{Around the {D}avenport--{H}eilbronn
  function}, Russ. Math. Surv. \textbf{66} (2011), no.~2, 221--270.

\bibitem{bombierihejhal}
Enrico Bombieri and Dennis~A. Hejhal, \emph{On the distribution of zeros of
  linear combinations of {E}uler products}, Duke Math. J. \textbf{80} (1995),
  no.~3, 821--862.

\bibitem{booker}
Andrew Booker and Frank Thorne, \emph{Zeros of {$L$}-functions outside the
  critical strip}, Algebra Number Theory \textbf{8} (2014), no.~9, 2027--2042.

\bibitem{cassels}
J.W.S. Cassels, \emph{Footnote to a note of {D}avenport and {H}eilbronn}, J.
  London Math. Soc. \textbf{36} (1961), 177--184.

\bibitem{conreyghosh}
John~Brian Conrey and Amit Ghosh, \emph{Tur\'an inequalities and zeros of
  {D}irichlet series associated with certain cusp forms}, Trans. Amer. Math.
  Soc. \textbf{342} (1994), no.~1, 407--419.

\bibitem{davenport1}
Harold Davenport and Hans Heilbronn, \emph{On the zeros of certain {D}irichlet
  series}, J. London Math. Soc. \textbf{11} (1936), 181--185.

\bibitem{iwaniec}
Henryk Iwaniec and Peter Sarnak, \emph{Perspectives on the analytic theory of
  {$L$}-functions}, {GAFA 2000. Visions in mathematics---Towards 2000.
  Proceedings of a meeting, Tel Aviv, Israel, August 25--September 3, 1999.
  Part II}, Birkh\"auser, special volume of the journal geometric and
  functional analysis ed., 2000, pp.~705--741.

\bibitem{kac}
Mark Kac, \emph{Statistical independence in probability, analysis and number
  theory.}, The Carus Mathematical Monographs, No. 12, Published by the
  Mathematical Association of America. Distributed by John Wiley and Sons,
  Inc., New York, 1959.

\bibitem{kaczorowski}
Jerzy Kaczorowski, \emph{Axiomatic theory of {$L$}-functions: the {S}elberg
  class}, Analytic number theory, Lecture Notes in Math., vol. 1891, Springer,
  Berlin, 2006, pp.~133--209.

\bibitem{kulas}
Jerzy Kaczorowski and Mieczys{\l}aw Kulas, \emph{On the non-trivial zeros off
  the critical line for {$L$}-functions from the extended {S}elberg class},
  Monatsh. Math. \textbf{150} (2007), no.~3, 217--232.

\bibitem{kmp1}
Jerzy Kaczorowski, Giuseppe Molteni, and Alberto Perelli, \emph{Linear
  independence of {$L$}-functions}, Forum Math. \textbf{18} (2006), no.~1,
  1--7.

\bibitem{kmp2}
\bysame, \emph{Some remarks on the unique factorization in certain semigroups
  of classical {$L$}-functions}, Funct. Approximatio, Comment. Math.
  \textbf{37} (2007), 263--275.

\bibitem{kaczper1}
Jerzy Kaczorowski and Alberto Perelli, \emph{The {S}elberg class: a survey},
  Number theory in progress, {V}ol. 2 ({Z}akopane-{K}o\'scielisko, 1997), de
  Gruyter, Berlin, 1999, pp.~953--992.

\bibitem{kaczper3}
\bysame, \emph{A note on the degree conjecture for the {S}elberg class}, Rend.
  Circolo Mat. Palermo \textbf{57} (2008), 443--448.

\bibitem{liu2}
Jianya Liu and Yangbo Ye, \emph{Selberg's orthogonality conjecture for
  automorphic {$L$}-functions}, Am. J. Math. \textbf{127} (2005), no.~4,
  837--849.

\bibitem{nakamurapankowski}
Takashi Nakamura and {\L}ukasz Pa\'{n}kowski, \emph{Value distribution for the
  derivatives of the logarithm of {$L$}-functions from the selberg class in the
  half-plane of absolute convergence}, arXiv:1501.02045, 2015.

\bibitem{neukirch}
J\"urgen Neukirch, \emph{Class field theory}, Grundlehren der Mathematischen
  Wissenschaften, no. 280, Springer-Verlag, 1986.

\bibitem{perelli2}
Alberto Perelli, \emph{A survey of the {S}elberg class of {$L$}-functions.
  {II}}, Riv. Mat. Univ. Parma (7) \textbf{3*} (2004), 83--118.

\bibitem{perelli1}
\bysame, \emph{A survey of the {S}elberg class of {$L$}-functions. {I}}, Milan
  J. Math. \textbf{73} (2005), 19--52.

\bibitem{rudnick}
Ze{\'e}v Rudnick and Peter Sarnak, \emph{Zeros of principal {$L$}-functions and
  random matrix theory}, Duke Math. J. \textbf{81} (1996), no.~2, 269--322.

\bibitem{saias}
Eric Saias and Andreas Weingartner, \emph{Zeros of {D}irichlet series with
  periodic coefficients}, Acta Arith. \textbf{140} (2009), no.~4, 335--344.

\bibitem{selberg}
Atle Selberg, \emph{Old and new conjectures and results about a class of
  {D}irichlet series}, Proceedings of the {A}malfi conference on analytic
  number theory, held at {M}aiori, {A}malfi, {I}taly, from 25 to 29 September,
  1989, Salerno: Universit\`a di Salerno, 1992, pp.~367--385.

\bibitem{tenenbaum}
G\'erald Tenenbaum, \emph{Introduction to analytic and probabilistic number
  theory}, Cambridge Studies in Advanced Mathematics, vol.~46, Cambridge Univ.
  Press, 1995.

\end{thebibliography}

\end{document}